\documentclass[a4paper,12pt,intlimits,oneside,reqno]{amsart}
\usepackage{amssymb}
\usepackage{amssymb, amsmath,txfonts}
\usepackage{tikz}
\usepackage{multirow}
\usepackage{xcolor}
\usepackage[colorlinks=true]{hyperref}

\numberwithin{equation}{section}
\numberwithin{figure}{section}
\def\R{{\mathbb R}}
\def\C{{\mathbb C}}
\def\T{{\mathbb T}}

\newtheorem{definition}{Definition}[section]
\newtheorem{theorem}{Theorem}[section]
\newtheorem{lemma}{Lemma}[section]
\newtheorem{remark}{Remark}[section]
\newtheorem{corollary}{Corollary}[section]
\newtheorem{proposition}{Proposition}[section]

\newtheorem{example}{Example}[section]

\begin{document}
\title[The $\alpha$--Szeg\H{o} equation]{The cubic Szeg\H{o} equation with a linear perturbation}
\author{Haiyan XU}
\address{Universit\'e Paris-Sud XI, Laboratoire de Math\'ematiques
d'Orsay, CNRS, UMR 8628} \email{{\tt Haiyan.xu@math.u-psud.fr}}
\thanks{This work was supported by grants from R\'egion Ile-de-France (RDMath - IdF). }

\begin{abstract}
We consider the following Hamiltonian equation on the $L^2$ Hardy space on the circle $\mathbb{S}^1$,
\begin{equation}
i\partial _t u=\Pi(\vert u\vert ^2u)+\alpha(u|1)\ ,\alpha\in\mathbb{R}\ ,
\end{equation}
where $\Pi$ is the Szeg\H{o} projector. The above equation with $\alpha=0$ was introduced by G\'erard and Grellier as an important mathematical model \cite{GGASENS,GGINV,GG2010Seminar}. In this paper, we continue our studies started in \cite{XU2014APDE}, and prove our system is completely integrable in the Liouville sense. We study the motion of the singular values of the related Hankel operators and find a necessary condition of norm explosion. As a consequence, we prove that the trajectories of the solutions will stay in a compact subset, while more initial data will lead to norm explosion in the case $\alpha>0$.
\end{abstract}

\maketitle

\section{Introduction}
The purpose of this paper is to study the following Hamiltonian system,
\begin{equation}\label{alszc2}
i\partial_t u=\Pi(|u|^2u)+\alpha(u|1)\ ,\ x\in\mathbb{S}^1\ ,\ t\in\R\ ,\ \alpha\in\mathbb{R}\ . 
\end{equation}
where the operator $\Pi$ is defined as a projector onto the non-negative frequencies, which is called the Szeg\H{o} projector. When $\alpha=0$, the equation above turns out to be the cubic Szeg\H{o} equation, 
\begin{equation}\label{szegoc2}
i\partial_t u=\Pi(|u|^2u)\ ,
\end{equation}
which was introduced by P. G\'erard and S. Grellier as an important mathematical model of the completely integrable systems and non-dispersive dynamics \cite{GGASENS, GGINV}. For $\alpha\ne0$, by changing variables as $u=\sqrt{|\alpha|}\tilde{u}(|\alpha| t)$, then $\tilde{u}$ satisfies
\begin{equation}
i\partial_t \tilde{u}=\Pi(|\tilde{u}|^2\tilde{u})+\mathrm{sgn}(\alpha)(\tilde{u}|1)\ .
\end{equation}
Thus our target equation with $\alpha\neq0 $ becomes
\begin{equation}\label{alszegoc2}
i\partial_t u=\Pi(|u|^2u)\pm(u|1)\ . 
\end{equation}

\subsection{Lax Pair structure}
Thanks to the Lax pairs for the cubic Szeg\H{o} equation \eqref{szegoc2} \cite{GGINV}, we are able to find a Lax pair for \eqref{alszc2}. To introduce the Lax pair structure, let us first define some useful operators and notation. For $X\subset\mathcal{D}'(\mathbb{S}^1)$, we denote
\begin{equation}
 X_+(\mathbb{S}^1):=\Big\{u({\mathrm{e}}^{i\theta})\in X,\ u({\mathrm{e}}^{i\theta})=\sum\limits_{k\ge0}\hat{u}(k){\mathrm{e}}^{ik\theta}\ \Big\}\ .
\end{equation}
For example, $L^2_+$ denotes the Hardy space of $L^2$ functions which extend to the unit disc $\mathbb{D}=\{z\in\C,\ |z|<1\}$ as holomorphic functions
\begin{equation}
 u(z)=\sum\limits_{k\ge0}\hat{u}(k)z^k,\ \sum\limits_{k\ge0}|\hat{u}(k)|^2<\infty\ .
\end{equation}
Then  the Szeg\H{o} operator $\Pi$ is an orthogonal projector $L^2(\mathbb{S}^1)\to L^2_+(\mathbb{S}^1)$.

Now, we are to define a Hankel operator and a Toeplitz operator. By a Hankel operator we mean a bounded operator $\Gamma$ on the sequence space $\ell^2$ which has a Hankel matrix in the standard basis $\{e_j\}_{j\geq0}$,
\begin{equation}
(\Gamma e_j,\ e_k\ )=\gamma_{j+k},\ j,\ k\ge0\ ,
\end{equation}
where $\{\gamma_j\}_{j\ge0}$ is a sequence of complex numbers. More backgrounds on the Hankel operators can be found in \cite{Peller2003}.

Let $S$ be the shift operator on $\ell^2$,
\begin{equation*}
S e_j=e_{j+1},\ j\ge0\ .
\end{equation*}
It is easy to show that a bounded operator $\Gamma$ on $\ell^2$ is a Hankel operator if and only if
\begin{equation}
 S^*\Gamma=\Gamma S\ .
\end{equation}

\begin{definition}
For any given $u\in H^{\frac12}_+(\mathbb{S}^1)$, $b\in L^\infty(\mathbb{S}^1)$, we define two operators $H_u, \ T_b:\ L^2_+\to L^2_+$ as follows. For any $h\in L^2_+$,
\begin{align}
H_u(h)=\Pi(u\bar{h})\ ,\\
T_b(h)=\Pi(bh)\ .
\end{align}
\end{definition}
Notice that $H_u$ is $\C-$antilinear and symmetric with respect to the real scalar product $\mathrm{Re}(u|v)$. In fact, it satisfies
\begin{equation*}
(H_u(h_1)|h_2)=(H_u(h_2)|h_1)\ .
\end{equation*}
 $T_b$ is $\C-$linear and is self-adjoint if and only if $b$ is real-valued. 
 
Moreover, $H_u$ is a Hankel operator. Indeed, it is given in terms of Fourier coefficients by
\begin{equation*}
\widehat{H_u(h)}(k)=\sum_{\ell\ge0}\hat{u}(k+\ell)\overline{\hat{h}(\ell)}\ ,
\end{equation*}
then 
\begin{eqnarray*}
S^*H_u(h)&=&\sum_{k,\ell\ge0}\hat{u}(k+\ell)\overline{\hat{h}(\ell)}S^*e_k=\sum_{k,\ell\ge0}\hat{u}(k+\ell+1)\overline{\hat{h}(\ell)}e_k\ ,\\
H_uSh&=& \sum_{k\ge\ell,\ell\ge0}\hat{u}(k)e_{k}\overline{\hat{h}(\ell)e_{\ell+1}}= \sum_{k,\ell\ge0}\hat{u}(k+\ell+1)\overline{\hat{h}(\ell)}e_k\ ,
\end{eqnarray*}
which means $S^*H_u=H_uS$, thus $H_u$ is a Hankel operator. We may also represent $T_b$ in terms of Fourier coefficients,
\begin{equation*}
\widehat{T_b(h)}(k)=\sum_{\ell\ge0}\hat{b}(k-\ell)\hat{h}(\ell)\ ,
\end{equation*}
then its matrix representation, in the basis ${e_k, k \ge 0}$, has constant diagonals, $T_b$ is a Toeplitz operator.

We now define another operator $K_u:=T_z^*H_u$. In fact $T_z$ is exactly the shift operator $S$ as above, we then call $K_u$ the shifted Hankel operator, which satisfying the following identity
\begin{equation}\label{HKc2}
K_u^2=H_u^2-(\cdot\ |\ u)u\ .
\end{equation}

Using the operators above, G\'erard and Grellier found two Lax pairs for the Szeg\H{o} equation \eqref{szegoc2}.
\begin{theorem}\cite[Theorem 3.1]{GGASENS}
Let $u \in C(R, H^s_+(\mathbb{S}^1))$ for some $s > 1/2$. The cubic Szeg\H{o} equation \eqref{szegoc2} has two Lax pairs $(H_u, B_u)$ and $(K_u,C_u)$, namely, if $u$ solves \eqref{szegoc2}, then
\begin{equation}
\frac{d H_u}{dt}=[B_u,H_u]\ , \frac{d K_u}{dt}=[C_u,K_u]\ ,
\end{equation}
where
$$B_u:=\frac{i}{2}H_u^2-iT_{|u|^2}\ , \ C_u=\frac{i}{2}K_u^2-iT_{|u|^2}\ .$$
\end{theorem}

For $\alpha\neq0$, the perturbed Szeg\H{o} equation \eqref{alszc2} is globally well-posed and by simple calculus, we find that $(H_u,B_u)$ is no longer a Lax pair, in fact,
\begin{equation}
\label{laxhuc2}\frac{d H_u}{dt}=[B_u,H_u]-i\alpha(u|1)H_1\ .
\end{equation}
Fortunately, $(K_u,C_u)$ is still a Lax pair.
\begin{theorem}\cite{XU2014APDE}\label{szlaxc2}
Given $u_0\in H^{\frac12}_+(\mathbb{S}^1)$, there exists a unique global solution $u\in C(\R;H^{\frac12}_+)$ of \eqref{alszc2} with $u_0$ as the initial condition. Moreover, if $u_0\in H^s_+(\mathbb{S}^1)$ for some $s>\frac12$, then $u\in C^\infty(\R;H^s_+)$. Furthermore, the perturbed Szeg\H{o} equation \eqref{alszc2} has a Lax pair $(K_u, C_u)$, namely, if $u$ solves \eqref{alszc2}, then
\begin{equation}\label{laxc2}
 \frac{d K_u}{dt}=[C_u,K_u]\ .
\end{equation}
\end{theorem}

An important consequence of this structure is that, if $u$ is a solution of \eqref{alszc2}, then $K_{u(t)}$ is unitarily equivalent to $K_{u_0}$. In particular, the spectrum of the $\mathbb{C} $-linear positive self-adjoint trace class operator $K_u^2$ is conserved by the evolution.

Denote
\begin{equation}
\mathcal{L}(N):=\left\{u: {\mathrm{rk}}(K_u)=N ,N\in\mathbb{N}^+\right\}\ .
\end{equation}

Thanks to the Lax pair structure, the manifolds $\mathcal{L}(N)$ are invariant under the flow of \eqref{alszc2}. Moreover, they turn out to be spaces of rational functions as in the following Kronecker type theorem.
\begin{theorem}\label{kroneckerc2}\cite{XU2014APDE}
$u\in \mathcal{L}(N)$ if and only if $u(z)=\frac{A(z)}{B(z)}$ is a rational function with
$$A,B\in\mathbb{C}_N[z],A\wedge B=1, \deg( A)= N \text{ or }
\deg( B)= N, B^{-1}(\{0\})\cap \overline{\mathbb D}=\emptyset\ ,$$
where $A\wedge B=1$ means $A$ and $B$ have no common factors.
\end{theorem}

Our main objective of the study on this mathematical model \eqref{alszc2} is on the large time unboundedness of the solution. This general question of existence of unbounded Sobolev trajectories comes back to \cite{BourgainGAFA}, and was addressed by several authors for various Hamiltonian PDEs, see e.g. \cite{CKSTT2010,GGHW,Guardia,GHP,GK2015,BGLT,Hani,HPTV,HaniThomann,
HausProcesi,Pocovnicu2011DCDS}. We have already considered the case with initial data $u_0\in\mathcal{L}(1)$ and found that
\begin{theorem}\cite{XU2014APDE}
Let $u$ be a solution to the $\alpha$--Szeg\H{o} equation,
\begin{equation}\left\{\begin{split}
&i\partial_t u=\Pi(|u|^2u)+\alpha(u|1)\ ,\ \alpha=\R\ ,\\
&u(0,x)=u_0(x)\in\mathcal{L}(1)\ .\
\end{split}\right.
\end{equation}

For $\alpha<0$, the Sobolev norm of the solution will stay bounded, uniform if $u_0$ is in some compact subset of $\mathcal{L}(1)$,
$$\|u(t)\|_{H^s}\le C\ ,\text{ $C$ does not depend on time $t$ , }s\geq0\ .$$

For $\alpha>0$, the solution $u$ of the $\alpha$--Szeg\H{o} equation has an exponential-on-time Sobolev norm growth,
\begin{equation}
\|u(t)\|_{H^s}\simeq {\mathrm{e}}^{C_{s}|t|}\ ,\ s>\frac12\ ,\ C_{s}>0\ ,\ |t| \rightarrow \infty\ ,
\end{equation}
if and only if
\begin{equation}
E_\alpha=\frac14Q^2+\frac12Q ,
\end{equation}
with $E_\alpha$ and $Q$ as the two conserved quantities, energy and mass.
\end{theorem}

\subsection{Main results}
We continue our studies on the cubic Szeg\H{o} equation with a linear perturbation \eqref{alszc2} on the circle $\mathbb{S}^1$ with more general initial data $u_0\in\mathcal{L}(N)$ for any $N\in\mathbb{N}^+$.

Firstly, the system is integrable since there are a large amount of conservation laws which comes from the Lax pair structure\eqref{laxc2}.
\begin{theorem}
Let $u(t,x)$ be a solution of \eqref{alszc2}. For every Borel function $f$ on $\R $, the following quantity 
$$L_f(u):=\Big(f(K_u^2)u\vert u\Big)-\alpha \Big(f(K_u^2)1\vert 1\Big)$$
is conserved.
\end{theorem}

Let $\sigma_k^2$ be an eigenvalue of $K_u^2$, and $f$ be the characteristic function of the singleton $ \{ \sigma _k^2\} $, then
$$\ell_k(u):=\Vert u'_k\Vert^2-\alpha\Vert v'_k\Vert^2$$
is conserved, where $u'_k$, $v'_k$ are the projections of $u$ and $1$ onto $\ker (K_u^2-\sigma_k^2)$, and $\Vert\cdot\Vert$ denotes the $L^2$--norm on the circle. Generically, on the $2N+1$--dimensional complex manifold $\mathcal{L}(N)$, we have $2N+1$ linearly independent and in involution conservation laws, which are $\sigma_k\ ,\ 1\leq k\leq N$ and $\ell_m\ ,\ 0\leq m\leq N$. Thus, the system \eqref{alszc2} can be approximated by a sequence of systems of finite dimension which are completely integrable in the Liouville sense.

Secondly, we prove the existence of unbounded trajectories for data in $\mathcal{L}(N)$ for any arbitrary $N\in\mathbb{N}^+$. One way to capture the unbounded trajectories of solutions is via the motion of singular values of $H_u^2$ and $K_u^2$. In the case with $\alpha=0$, all the eigenvalues of $H_u^2$ and $K_u^2$ are constants, but the eigenvalues of $H_u^2$ are no longer constants for $\alpha\neq0$, which makes the system more complicated. 

By studying the motion of singular values of $H_u$ and $K_u$, we gain that the necessary condition and existence of crossing which means the two closest eigenvalues of $H_u$ touch some eigenvalue of $K_u$ at some finite time. A remarkable observation is that the Blaschke products of $K_u$ never change their $\mathbb{S}^1$ orbits as time goes.

The main result on the large time behaviour of solutions is as below.
\begin{theorem}\label{mainthmc2}
Let $u_0\in\mathcal{L}(N)$ for any $N\in\mathbb{N}^+$. 

If $\alpha<0$, the trajectory of the solution $u(t)$ of the $\alpha$--Szeg\H{o} \eqref{alszc2} stays in a compact subset of $\mathcal{L}(N)$. In other words, the Sobolev norm of the solution $u(t)$ will stay bounded,
$$\|u(t)\|_{H^s}\le C\ ,\text{ $C$ does not depend on time $t$ , }s\ge0\ .$$
While for $\alpha>0$, there exists $u_0\in \mathcal{L}(N)$ which leads to a solution with norm explosion at infinity. More precisely, 
$$\Vert u(t)\Vert_{H^s}\simeq \mathrm{e}^{C_{\alpha}(2s-1)|t|}\ ,\ t\to\infty\ ,\ \forall s\geq\frac12\ .$$
\end{theorem}

\begin{remark}\mbox{}\\
\noindent{\bf 1. }In the case $\alpha=0$, there are two Lax pairs, the conserved quantities are much simpler, which are the eigenvalues of $H_u^2$ and $K_u^2$. While in the case $\alpha\ne0$, the eigenvalues of $H_u^2$ are no longer conserved, which makes our system more complicated. 

\noindent {\bf 2. }For the cubic Szeg\H{o} equation with $\alpha=0$, G\'erard and Grellier \cite{GG2015arxiv} have proved there exists a $G_\delta$ dense set $\mathfrak{g}$ of initial data in $ C^\infty_+:=\cap_s H^s$, such that for any $v_0\in\mathfrak{g}$, there exist sequences of time $\overline{t_n}$ and $\underline{t_n}$, such that the corresponding solution $v$ of the cubic Szeg\H{o} equation
\begin{equation}\label{szegocubic}
i\partial_tv=\Pi_+(|v|^2v)\ ,\ v(0)=v_0\ ,
\end{equation} 
satisfies
\begin{equation}\label{forward}
\forall r>\frac12\ ,\ \forall M\geq1\ , \frac{\Vert v(\overline{t_n})\Vert_{H^r}}{|\overline{t_n}|^M}\to\infty\ ,\  n\to\infty\ ,
\end{equation}
while 
\begin{equation}\label{inverse}
v(\underline{t_n})\to v_0 \text{ in }C_+^\infty\ ,\  n\to\infty\ .
\end{equation}
Here, by considering the rational data in the case $\alpha\ne0$, we proved the existence of solutions with exponential growth in time rather than $\limsup$. 

There is another non dispersive example with norm growth by Oana Pocovnicu \cite{Pocovnicu2011DCDS}, who studied the cubic Szeg\H{o} equation on the line $\mathbb{R}$, and found there exist solutions with Sobolev norms growing polynomially in time as $|t|^{2s-1}$ with $s\geq1/2$. 

\noindent {\bf 3. }For the case $\alpha>0$, we now have solutions of \eqref{alszc2} with different growths, uniformly bounded, growing in fluctuations with a $\limsup$ super-polynomial in time growth, and exponential in time growth. Indeed, it is easy to show that $zu(t,z^2)$ is a solution to the $\alpha$--Szeg\H{o} equation if $u(t,z)$ solves the cubic Szeg\H{o} equation \eqref{szegocubic}. Thus, for the cubic Szeg\H{o} equation with a linear perturbation \eqref{alszc2}, there also exist solutions with such an energy cascade as in \eqref{forward} and $\eqref{inverse}$. 

\noindent {\bf 4. }In this paper, we consider data in $\mathcal{L}(N)$ for any arbitrary $N\in\mathbb{N}^+$. The data we find which lead to a large time norm explosion are very special. An interesting observation is that the equations on $u'_k$ and $v'_k$ look similar to the original $\alpha$--Szeg\H{o} equation, 
\begin{equation}
\label{pksystem}
\frac{\partial}{\partial t}\begin{pmatrix}
u'_k \\
v'_k
\end{pmatrix}
=-i
\begin{pmatrix}
 T_{\vert u\vert^2} & \alpha(u\vert 1)\\
-(1\vert u) & T_{\vert u\vert^2}-\sigma_k^2
\end{pmatrix}
\quad
\begin{pmatrix}
u'_k \\
v'_k
\end{pmatrix}\ ,
\end{equation}
which gives us some hope to extend our results to general rational data.

\end{remark}

\subsection{Organization of this chapter}
In section 2, we recall the results about the singular values of $H_u$ and $K_u$ \cite{GGHankel}. In section 3, we introduce the conservation laws and prove the integrability. In section 4, we study the motion of the singular values of the Hankel operators $H_u$ and $K_u$, the eigenvalues of $H_u$ move and may touch some eigenvalue of $K_u$ at finite time while the eigenvalues of $K_u$ stay fixed with the corresponding Blaschke products stay in the same orbits. In section 5, we present a necessary condition of the norm explosion, and as a direct consequence, we know that for $\alpha<0$, the trajectories of the solutions stay in a compact subset. In section 6, we study the norm explosion with $\alpha>0$ for data in $\mathcal{L}(N)$ with any $N\in\mathbb{N}^+$. We present some open problems in the last section.

\section{Spectral analysis of the operators $H_u$ and $K_u$}
In this section, let us introduce some notation which will be used frequently and some useful results by G\'erard and Grellier in their recent work \cite{GGHankel}. We consider $u\in H^s_+(\mathbb{S}^1)$ with $s>\frac12$. The Hankel operator $H_u$ is compact by the theorem due to Hartman \cite{Ha}. Let us introduce the spectral analysis of operators $H_u^2$ and $K_u^2$. For any $\tau\geq0$, we set
\begin{equation}
 E_u(\tau):=\ker (H_u^2-\tau^2\mathrm{I}),\ F_u(\tau):=\ker (K_u^2-\tau^2\mathrm{I})\ .
\end{equation}

If $\tau>0$, the $E_u(\tau)$ and $F_u(\tau)$ are finite dimensional with the following properties.
\begin{proposition}\label{rigidityc2}\cite{GGHankel}
Let $u\in H^s_+(\mathbb{S}^1 )\setminus \{ 0\} $ with $s>1/2$, and $\tau>0$ such that
\[E_u(\tau)\neq \{ 0\} \quad or\quad F_u(\tau)\neq \{ 0\}\ .\]Then one of the following properties holds.
\begin{enumerate}
\item $\dim E_u(\tau)=\dim F_u(\tau)+1$,  $u \not \perp E_u(\tau)$, and $F_u(\tau)=E_u(\tau)\cap u^\perp $.
\item $\dim F_u(\tau)=\dim E_u(\tau)+1$,  $u \not \perp F_u(\tau)$, and $E_u(\tau)=F_u(\tau)\cap u^\perp $.
\end{enumerate}
Moreover, if $u_\rho$ and $u'_\sigma$ denote respectively the orthogonal projections of $u$ onto $E_u(\rho)$, $\rho\in \Sigma_H(u)$, and onto $F_u(\sigma)$, $\sigma\in \Sigma_K(u)$ with
$$\Sigma_H(u):=\{\tau>0:\ u\not\perp E_u(\tau)\},\quad\Sigma_K(u):=\{\tau\geq0:\ u\not\perp F_u(\tau)\}\ .$$
Then
\begin{enumerate}
\item $\Sigma_H(u)$ and $\Sigma_K(u)$ are disjoint, with the same cardinality;
\item if $\rho\in \Sigma_H(u)$, 
\begin{equation}
u_\rho=\|u_\rho\|^2\sum_{\sigma\in \Sigma_K(u)}\frac{u'_\sigma}{\rho^2-\sigma^2}\ ,
\end{equation}
 \item if $\sigma\in \Sigma_K(u)$,
 \begin{equation}
 u'_\sigma=\|u'_\sigma\|^2 \sum_{\rho\in \Sigma_H(u)}\frac{u_\rho}{\rho^2-\sigma^2}\ .
 \end{equation}
\item A non negative number $\sigma$ belongs to $\Sigma _K(u)$ if and only if it does not belong to $\Sigma _H(u)$ and
\begin{equation}
\sum_{\rho \in \Sigma _H(u)}\frac{\|u_\rho \|^2}{\rho ^2-\sigma^2}=1\ .
\end{equation}
\end{enumerate}
\end{proposition}

By the spectral theorem for $H_u^2$ and $K_u^2$, which are self-adjoint and compact, we have the following orthogonal decomposition
\begin{equation}
L_+^2=\overline{\oplus_{\tau>0}E_u(\tau)}=\overline{\oplus_{\tau\ge0}F_u(\tau)}\ .
\end{equation}
Then we can write $u$ as
\begin{equation}
u=\sum\limits_{\rho\in\Sigma_H(u)}u_\rho=\sum\limits_{\sigma\in\Sigma_K(u)}u'_\sigma\ .
\end{equation}

In fact, we are able to describe these two sets $E_u(\tau)$ and $F_u(\tau)$ more explicitly. Recall that a finite Blaschke product of degree $k$ is a rational function of the form
$$\Psi (z)={\mathrm{e}}^{-i\psi}\frac{P(z)}{D(z)}\ ,$$
where $\psi \in \mathbb{S}^1 $ is called the angle of $\Psi$ and $P$ is a monic polynomial of degree $k$ with all its roots in $\mathbb D $, $D(z)=z^k\overline P\left (\frac 1z\right )$ as the normalized denominator of $\Psi$. Here a monic polynomial is a univariate polynomial in which the leading coefficient (the nonzero coefficient of highest degree) is equal to $1$. We denote by $\mathcal{B}_k$ the set of all the Blaschke functions of degree $k$.

\begin{proposition}\label{actionc2}\cite{GGHankel}
Let $\tau>0$ and $u\in H^s_+(\mathbb{S}^1 )$ with $s>\frac12$.
\begin{enumerate}
\item Assume $\tau\in \Sigma_H(u)$ and $\ell:=\dim E_u(\tau)=\dim F_u(\tau)+1$. Denote by $u_\tau$ the orthogonal projection of $u$ onto $E_u(\tau)$. There exists a Blaschke function $\Psi _\tau\in \mathcal B_{\ell-1}$ 
such that
$$\tau u_\tau=\Psi _\tau H_u(u_\tau)\ ,$$
and if $D$ denotes the normalized denominator of $\Psi _\tau$,
\begin{eqnarray}
E_u(\tau)&=&\left \{ \frac f{D(z)} H_u(u_\tau)\ ,\ f\in \mathbb{C} _{\ell-1}[z]\right \} \ ,\\
F_u(\tau)&=&\left\{ \frac g{D(z)} H_u(u_\tau)\ ,\ g\in \mathbb{C} _{\ell-2}[z]\right \},
\end{eqnarray}
and for $a=0,\dots,\ell-1\ ,\ b=0,\dots ,\ell-2$,
\begin{eqnarray} 
H_u\left (\frac{z^a}{D(z)}H_u(u_\tau)\right )&=&\tau{\mathrm{e}}^{-i\psi _\tau}\frac{z^{\ell-a-1}}{D(z)}H_u(u_\tau)\ , \\
K_u\left (\frac{z^b}{D(z)}H_u(u_\tau)\right )&=&\tau{\mathrm{e}}^{-i\psi _\tau}\frac{z^{\ell-b-2}}{D(z)}H_u(u_\tau)\ ,
\end{eqnarray}
where $\psi _\tau$ denotes the angle of $\Psi _\tau$.
\item Assume $\tau\in \Sigma_K(u)$ and $m :=\dim F_u(\tau)=\dim E_u(\tau)+1$. Denote by $u_\tau'$ the orthogonal projection of $u$ onto $F_u(\tau)$. There exists an inner function $\Psi _\tau\in \mathcal B_{m -1}$ such that
$$K_u(u_\tau')=\tau\Psi _\tau u_\tau'\ ,$$
and if $D$ denotes the normalized denominator of $\Psi _\tau$,
\begin{eqnarray} 
F_u(\tau)&=&\left \{ \frac f{D(z)} u_\tau'\ ,\ f\in \mathbb{C} _{m -1}[z]\right \} \ ,\\
E_u(\tau)&=&\left \{ \frac {zg}{D(z)} u_\tau'\ ,\ g\in \mathbb{C} _{m -2}[z]\right \}\ ,
\end{eqnarray}
and, for $a=0,\dots,m -1\ ,\ b=0,\dots ,m-2$,
\begin{eqnarray} 
K_u\left (\frac {z^a}{D(z)}u_\tau'\right )&=&\tau{\mathrm{e}}^{-i\psi _\tau}\frac {z^{m -a-1}}{D(z)}u_\tau'\ , \\
H_u\left (\frac {z^{b+1}}{D(z)}u_\tau'\right )&=&\tau{\mathrm{e}}^{-i\psi _\tau}\frac {z^{m-b-1}}{D(z)}u_\tau'\ ,
\end{eqnarray}
where $\psi _\tau$ denotes the angle of $\Psi _\tau$.
\end{enumerate}
\end{proposition}

We call the elements $\rho_j\in\Sigma_H(u)$ and $\sigma_k\in\Sigma_K(u)$ as the dominant eigenvalues of $H_u$ and $K_u$ respectively. Due to the above achievements, they are in a finite or infinite sequence
\[\rho_1>\sigma_1>\rho_2>\sigma_2>\cdots\to0\ ,\]
we denote by  $\ell_j$ and $m_k$ as the multiplicities of $\rho_j$ and $\sigma_k$ respectively. In other words,
\begin{eqnarray*}
\dim E_u(\rho_j)=\ell_j\ ,\\
\dim F_u(\sigma_k)=m_k\ .
\end{eqnarray*}
Therefore, we may define the dominant ranks of the operators as 
\begin{align*}
\mathrm{rk_d}(H_u):=\sum\limits_j \ell_j\ ,\\
\mathrm{rk_d}(K_u):=\sum\limits_k m_k\ ,
\end{align*}
while the ranks of the operators are
\begin{align*}
\mathrm{rk}(H_u)=\sum\limits_j \ell_j+\sum\limits_k (m_k-1)\ ,\\
\mathrm{rk}(K_u)=\sum\limits_j (\ell_j-1)+\sum\limits_k m_k\ .
\end{align*}

In this paper, $u_j$ and $u'_k$ denote the orthogonal projections of $u$ onto $E_u(\rho_j)$ and $F_u(\sigma_k)$ respectively, while $v_j$ and $v'_k$ denote the orthogonal projections of $1$ onto $E_u(\rho_j)$ and $F_u(\sigma_k)$. The $L^2$--norms of $u_j$ and $u'_k$ can be represented in terms of $\rho_\ell$'s and $\sigma_\ell$'s, which was already observed in \cite{GG2014Hankel}.
\begin{lemma}
Let $u\in H^{\frac12}(\mathbb{S}^1)$, $\Sigma_H(u)=\{\rho_j\}$ and $\Sigma_K(u)=\{\sigma_k\}$ with
$$\rho_1>\sigma_1>\rho_2>\cdots\ge0\ .$$
Then
\begin{align*}
\Vert u_j\Vert^2=\frac{\prod\limits_{\ell}(\rho_j^2-\sigma_\ell^2)}{\prod\limits_{\ell\neq j}(\rho_j^2-\rho_\ell^2)}\ ,\ \Vert u'_k \Vert ^2=\frac{\prod\limits_\ell (\rho _\ell^2-\sigma_k ^2)}{\prod\limits_{\ell\neq k} (\sigma_\ell ^2-\sigma_k^2)}\ .
\end{align*}
\end{lemma}

\begin{proof}
First, we have
\begin{equation*}
\big((I-xH_u^2)^{-1}1\ |\ 1\big)=\prod\limits_\ell\frac{1-x\sigma_\ell^2}{1-x\rho_\ell^2}\ .
\end{equation*}
In fact, we can rewrite the left hand side as
\begin{equation*}
\big((I-xH_u^2)^{-1}1\ |\ 1\big)=\sum\limits_\ell \frac{\Vert v_\ell\Vert^2}{1-x\rho_\ell^2}+1-\sum_\ell\Vert v_\ell\Vert^2\ .
\end{equation*}
From Proposition~\ref{actionc2},
\begin{equation*}
v_j=\Big(1,\frac{H_u(u_j)}{\Vert H_u(u_j)\Vert}\Big)\frac{H_u(u_j)}{\Vert H_u(u_j)\Vert}\ ,
\end{equation*}
combined with $\Psi_jH_u(u_j)=\rho_ju_j$,
we get
\begin{equation*}
\Vert v_j\Vert^2=\frac{|(1,H_u(u_j))|^2}{\Vert H_u(u_j)\Vert^2}=\frac{|(H_u(1),u_j)|^2}{\rho_j^2\Vert u_j\Vert^2}=\frac{\Vert u_j\Vert^2}{\rho_j^2}\ .
\end{equation*}
Thus
\begin{equation*}
\prod\limits_\ell\frac{1-x\sigma_\ell^2}{1-x\rho_\ell^2}=\sum\limits_\ell \frac{\|u_\ell\|^2}{\rho_\ell^2(1-x\rho_\ell^2)}+1-\sum\frac{\|u_\ell\|^2}{\rho_\ell^2}\ .
\end{equation*}
We get, identifying the residues at $x=1/\rho_j^2$,
\begin{equation}
\label{normujc2}\|u_j\|^2=\frac{\prod\limits_{\ell}(\rho_j^2-\sigma_\ell^2)}{\prod\limits_{\ell\ne j}(\rho_j^2-\rho_\ell^2)}\ .
\end{equation}
On the other hand, since
\begin{equation*}
1-x((I-xK_u^2)^{-1}u\ |\ u)=\frac{1}{((I-xH_u^2)^{-1}1\ |\ 1)}\ ,
\end{equation*}
then
\begin{equation*}
1-x\Big(\sum\limits_k\frac{\|u'_k\|^2}{1-x\sigma_k^2}+\|u\|^2-\sum_k\|u'_k\|^2\Big)=\prod\limits_\ell\frac{1-x\rho_\ell^2}{ 1-x\sigma_\ell^2}\ ,
\end{equation*}
we get, identifying the residues at $x=1/\sigma_k^2$,
\begin{equation}
\label{normukc2}\| u'_k \| ^2=\frac{\prod\limits_\ell (\rho _\ell^2-\sigma_k ^2)}{\prod\limits_{\ell\ne k} (\sigma_\ell ^2-\sigma_k^2)}\ .
\end{equation}
\end{proof}
\section{Conservation laws and the $\alpha$--Szeg\H{o} hierarchy}
We endow $L^2_+(\mathbb{S}^1)$ with the symplectic form
$$\omega(u,v)=4\mathrm{Im}(u\ |\ v)\ .$$
Then \eqref{alszc2} can be rewritten as
\begin{equation}
\partial_t u=X_{E_\alpha}(u)\ ,
\end{equation}
with $X_{E_\alpha}$ as the Hamiltonian vector field associated to the Hamiltonian function given by
$$E_\alpha(u):=\frac{1}{4}\int_{\mathbb{S}^1}|u|^4\frac{d\theta}{2\pi}+\frac\alpha2|(u|1)|^2\ .$$
The invariance by translation and by multiplication by complex numbers of modulus $1$ gives two other formal conservation laws
\begin{align*}
&\text{mass:  }&Q(u):=\int_{\mathbb{S}^1}|u|^2 \frac{d\theta}{2\pi}=\|u\|_{L^2}^2\ ,\\
&\text{momentum:  }&M(u):=(Du|u),\ D:=-i\partial_\theta=z\partial_z\ .
\end{align*}
Moreover, the Lax pair structure leads to the conservation of the eigenvalues of $K_u^2$. So it is obvious the system is completely integrable for the data in the $3-$dimensional complex manifold $\mathcal{L}(1)$. Then what about the general case, for example in $\mathcal{L}(N)$ with arbitrary $N\in\mathbb{N}^+$? Fortunately, we are able to find many more conservation laws by its Lax pair structure \eqref{laxc2}. We will then show our system is still completely integrable with data in $\mathcal{L}(N)$ in the Liouville sense.
\subsection{Conservation laws}
Thanks to the Lax pair structure, we are able to find an infinite sequence of conservation laws.
\begin{theorem}\label{conservec2}
For every Borel function $f$ on $\R $, the following quantity 
$$L_f(u):=\Big(f(K_u^2)u\vert u\Big)-\alpha \Big(f(K_u^2)1\vert 1\Big)$$
is a conservation law.
\end{theorem}
\begin{proof}
From the Lax pair identity
$$\frac{dK_u}{dt}=[C_u,K_u]\ ,\ C_u=-iT_{\vert u\vert ^2}+\frac{i}{2}K_u^2\ ,$$
we infer
$$\frac{d}{dt}K_u^2=[-iT_{\vert u\vert ^2},K_u^2]\ ,$$
and consequently, for every Borel function $f$ on $\R $,
$$\frac{d}{dt}f(K_u^2)=[-iT_{\vert u\vert ^2},f(K_u^2)]\ .$$
On the other hand, the equation reads
$$\frac{d}{dt} u=-iT_{\vert u\vert ^2}u-i\alpha (u\vert 1)\ .$$
Therefore we obtain
\begin{eqnarray*}
\frac{d}{dt}\Big(f(K_u^2)u\vert u\Big)&=&\Big([-iT_{\vert u\vert ^2},f(K_u^2)]u\vert u\Big)-i\Big(f(K_u^2)T_{\vert u\vert ^2}u\vert u\Big)+i\Big(u\vert f(K_u^2)T_{\vert u\vert ^2}u\Big)\\
&& -i\alpha (u\vert 1)\Big(f(K_u^2)(1)\vert u\Big)+i\alpha (1\vert u)\Big(f(K_u^2)(u)\vert 1\Big)\\
&=&-i\alpha \Big[\big(f(K_u^2)(1)\vert (1\vert u)u\big)-\big((1\vert u)u\vert f(K_u^2)(1)\big)\Big]\ .
\end{eqnarray*}
Now observe that
$$(1\vert u)u=H_u^2(1)-K_u^2(1)=T_{\vert u\vert ^2}(1)-K_u^2(1)\ .$$
We obtain
\begin{eqnarray*}
\frac{d}{dt}\Big(f(K_u^2)u\vert u\Big)&=&-i\alpha \Big[\big(f(K_u^2)(1)\vert T_{\vert u\vert ^2}(1)\big)-\big(T_{\vert u\vert ^2}(1)\vert f(K_u^2)(1)\big)\Big]\\
&=&\alpha \Big([-iT_{\vert u\vert ^2}, f(K_u^2)] (1)\vert 1\Big)\\
&=&\alpha \frac{d}{dt}\Big(f(K_u^2)(1)\vert 1\Big)\ .
\end{eqnarray*}
\end{proof}

\subsection{The $\alpha$--Szeg\H{o} hierarchy}
By the theorem above, for any $n\in\mathbb{N}$, $$L_n(u):=\big(K_u^{2n}(u)\ |\ u\big)-\alpha\big(K_u^{2n}(1)\ |\ 1\big)\ $$ is conserved. Then the manifold $\mathcal{L}(N)$ is of $2N+1-$ complex dimension and admits $2N+1$ conservation laws, which are 
$$\sigma_k,\ k=1,\cdots,N\text{ and } L_n(u),\ n=0,1,\cdots,N\ .$$
We are to show that all these conservation laws are in involve. Since the $\sigma_k$'s are constants, it is sufficient to show that all these $L_n$ satisfy the Poisson commutation relations
\begin{equation}\label{involvec2}
\{L_n,L_m\}=0\ .
\end{equation}

Let us begin with the following lemma which helps us better understand the conserved quantities. 
\begin{lemma}\label{JFc2}
Let $u\in H^{\frac12}(\mathbb{S}^1)$, $\Sigma_H(u)=\{\rho_j\}$ and $\Sigma_K(u)=\{\sigma_k\}$ with
$$\rho_1>\sigma_1>\rho_2>\cdots\ge0\ .$$
Denote 
\begin{eqnarray*}
J_x(u)&:=&\big((1-xH_u^2)^{-1}(1)\ |\ 1\big)\ ,\\
Z_x(u)&:=&\big(1\ |\ (1-xH_u^2)^{-1}(u)\big)\ ,\\
F_x(u)&:=&\big((1-xK_u^2)^{-1}(u)\ |\ u\big)\ ,\\
E_x(u)&:=&\big((1-xK_u^2)^{-1}(1)\ |\ 1\big)\ .
\end{eqnarray*}
Then
\begin{align}
F_x(u)&=\frac{J_x(u)-1}{xJ_x(u)}\label{relationJF}\ ,\\
E_x(u)&=J_x(u)-x\frac{\vert Z_x(u)\vert^2}{J_x(u)}\ .\label{relationEJZ}
\end{align}
\end{lemma}

\begin{proof}
Recall \eqref{HKc2}, for any $f\in H^{\frac12}$, we have
$$K_u^2f=H_u^2f-(f\ |\ u)u\ .$$
Denote 
\begin{equation}\label{defw}
w(f)=(1-xH_u^2)^{-1}(f)-(1-xK_u^2)^{-1}(f)\ ,
\end{equation}
then
\begin{align*}
w(f)&=x\Big(f\ |\ (1-xK_u^2)^{-1}(u)\Big)(1-xH_u^2)^{-1}(u)\\
&=x\Big(f\ |\ (1-xH_u^2)^{-1}(u)\Big)(1-xK_u^2)^{-1}(u)\ .
\end{align*}
We may observe the two vectors $(1-xH_u^2)^{-1}(u)$ and $(1-xK_u^2)^{-1}(u)$ are co-linear,
\begin{equation}
\label{KH}(1-xK_u^2)^{-1}(u)=A(1-xH_u^2)^{-1}(u),\ A\in\mathbb{R}\ .
\end{equation}

Let us choose $f=u$, then
\begin{equation}
\label{eqw} \Big(w(u)\ |\ u\Big)=(1-A)\Big((1-xH_u^2)^{-1}(u)\ |\ u\Big)=Ax\Big(u\ |\ (1-xH_u^2)^{-1}(u)\Big)^2\ .
\end{equation}

We are to calculate the factor $A$. Since
\begin{align*}
&x\Big(u\ |\ (1-xH_u^2)^{-1}(u)\Big)=x\Big(1\ |\ (1-xH_u^2)^{-1}H_u^2(1)\Big)\\
&\qquad=\sum\limits_{n\geq0}x^{n+1}\Big(H_u^{2(n+1)}(1)\ |\ 1\Big)=\sum\limits_{n\ge0}x^{n}\Big(H_u^{2n}(1)\ |\ 1\Big)-1=J_x-1\ ,
\end{align*}
thus \eqref{eqw} yields 
$$1-A=(J_x-1)A\ ,$$
which means
$$A=\frac{1}{J_x}\ .$$

So \eqref{KH} turns out to be
\begin{equation}\label{JF1}
(1-xK_u^2)^{-1}(u)=\frac{1}{J_x}(1-xH_u^2)^{-1}(u)\ ,
\end{equation}
then combined with the definition of $w(f)$, we have 
\begin{equation}\label{JF2}
(1-xH_u^2)^{-1}(f)-(1-xK_u^2)^{-1}(f)=\frac{x}{J_x}\Big(f\ |\ (1-xH_u^2)^{-1}(u)\Big)(1-xH_u^2)^{-1}(u)\ .
\end{equation}

Using the equality \eqref{JF1},
\begin{align*}
F_x&=\Big((1-xK_u^2)^{-1}(u)\ |\ u\Big)=\frac{1}{J(x)}\Big((1-xH_u^2)^{-1}(u)\ |\ u\Big)\\
&=\frac{1}{J(x)}\Big((1-xH_u^2)^{-1}H_u^2(1)\ |\ 1\Big)=\frac{J_x-1}{xJ_x}\ .
\end{align*}

Now, we turn to prove \eqref{relationEJZ}. Use again \eqref{defw} with $f=1$, 
\begin{eqnarray*}
\Big(w(1)|1\Big)&=&\Big((1-xH_u^2)^{-1}(1)-(1-xK_u^2)^{-1}(1)|1\Big)=J_x-E_x\\
&=&x\Big(1|1-xH_u^2)^{-1}(1)\Big)\Big((1-xK_u^2)^{-1}(1)|1\Big)=x\overline{Z_x}\Big((1-xK_u^2)^{-1}(u)|1\Big)\ ,
\end{eqnarray*}
plugging \eqref{KH},
\begin{equation*}
\Big((1-xK_u^2)^{-1}(u)|1\Big)=\frac{1}{J_x}\Big((1-xH_u^2)^{-1}(u)|1\Big)=\frac{Z_x}{J_x}\ ,
\end{equation*}
then
\begin{equation}
J_x-E_x=x\overline{Z_x}\frac{Z_x}{J_x}=x\frac{|Z_x|^2}{J_x}\ ,
\end{equation}
which leads to \eqref{relationEJZ}.
\end{proof}

\medskip
Now, we are ready to show the following cancellation for the Poisson brackets of the conservation laws.
\begin{theorem} 
For any $x\in\R$, we set
$$L_x(u)=\big((1-xK_u^2)^{-1}(u)\ |\ u\big)-\alpha\big((1-xK_u^2)^{-1}(1)\ |\ 1\big)\ ,$$
Then $L_x(u(t))$ is conserved, and for every $x,y$,
\begin{equation}
\{L_x,\ L_y\}=0\ .
\end{equation}
\end{theorem}
\begin{proof}
Using the previous Lemma, we may rewrite
\begin{equation}\label{formulaLXc2}
L_x=\frac1x\big(1-\frac1{J_x}\big)-\alpha E_x\ ,
\end{equation}
with
\begin{align*}
J_x(u)&:=\big((1-xH_u^2)^{-1}(1)\ |\ 1\big)=1+x\big((1-xH_u^2)^{-1}(u)\ |\ u\big)\ ,\\
E_x(u)&:=\big((1-xK_u^2)^{-1}(1)\ |\ 1\big)=J_x(u)-x\frac{\vert Z_x(u)\vert^2}{J_x(u)}\ ,\\
Z_x(u)&:=\big(1\ |\ (1-xH_u^2)^{-1}(u)\big)\ .
\end{align*}

Recall that the identity
\begin{equation}
\{J_x,\ J_y\}=0
\end{equation}
which was obtained in \cite[section 8]{GGASENS}. We then have
\begin{equation}\label{bracketc2}
\{L_x,\ L_y\}=\alpha\Big(\frac{y}{xJ_x^2J_y}\{J_x,\ \vert Z_y\vert^2\}-\frac{x}{yJ_y^2J_x}\{J_y,\ \vert Z_x\vert^2\}\Big)+\alpha^2\{E_x,\ E_y\}\ .
\end{equation}
Let us first prove that $\{E_x,\ E_y\}=0$. Notice that
\begin{equation}
E_x(u)=J_x(S^*u)\ ,
\end{equation}
therefore
\begin{equation*}
dE_x(u)\cdot h=dJ_x(S^*u)\cdot(S^*h)=\omega(S^*h,X_{J_x}(S^*u))=\omega(h,SX_{J_x}(S^*u))\ .
\end{equation*}
We conclude
\begin{equation*}
X_{E_x}(u)=SX_{J_x}(s^*U)\ ,
\end{equation*}
thus
\begin{align*}
\{E_x,\ E_y\}(u)&=dE_y(u)\cdot X_{E_x}(u)=dJ_y(S^*u)\cdot S^*S X_{J_x}(S^*u)\\
&=d J_y(S^*u)\cdot X_{J_x}(S^*u)=\{J_x,J_y\}(S^*u)=0\ .
\end{align*}

We now show that the coefficient of $\alpha$ in \eqref{bracketc2} vanishes identically. It is enough to work on the generic states of $\mathcal{L}(N)$, so we can use the coordinates
$$(\rho_1,\cdots,\rho_{N+1},\sigma_1,\cdots,\sigma_{N},\varphi_1,\cdots,\varphi_{N+1},\theta_1,\cdots,\theta_{N})$$
for which we recall that
\begin{equation*}
\omega=\sum_{j=1}^{N+1}  d(\frac{\rho_j^2}{2})\wedge d\varphi_j+\sum_{k=1}^{N}d(\frac{\sigma_k^2}{2})\wedge d\theta_k\ .
\end{equation*}
Moreover, we have
\begin{equation*}
\rho_ju_j={\mathrm{e}}^{-i\varphi_j}H_u(u_j)\ ,
\end{equation*}
therefore, 
\begin{equation*}
Z_x(u)=\sum_{j=1}^{N+1}\frac{\Vert u_j\Vert^2}{\rho_j(1-x\rho_j^2)}{\mathrm{e}}^{i\varphi_j}\ .
\end{equation*}
Since
\begin{equation*}
J_x(u)=\frac{\prod_{k=1}^{N}(1-x\sigma_k^2)}{\prod_{j=1}^{N+1}(1-x\rho_j^2)}\ ,
\end{equation*}
we know that
\begin{equation*}
\{J_x,\ \varphi_j\}=\frac{2xJ_x}{1-x\rho_j^2}\ ,
\end{equation*}
and we infer
\begin{equation}
\{J_x,\ Z_y\}=2ixJ_x\sum_{j=1}^{N+1}\frac{\Vert u_j\Vert^2}{\rho_j(1-x\rho_j^2)(1-y\rho_j^2)}{\mathrm{e}}^{i\varphi_j}=\frac{2ixJ_x}{x-y}(xZ_x-yZ_y)\ .
\end{equation}
Consequently,
\begin{equation}
\{J_x,\ \vert Z_y\vert^2\}=2{\mathrm{Re}}(\overline{Z}_y\{J_x,\ Z_y\})=-\frac{4x^2J_x}{x-y}{\mathrm{Im}}(\overline{Z}_yZ_x)\ .
\end{equation}
We conclude that
\begin{equation}
\frac{y}{xJ_x^2J_y}\{J_x,\ \vert Z_y\vert^2\}-\frac{x}{yJ_y^2J_x}\{J_y,\ \vert Z_x\vert^2\}=-\frac{4xy}{(x-y)J_xJ_y}\big({\mathrm{Im}}(\overline{Z}_yZ_x)+{\mathrm{Im}(\overline{Z}_xZ_y)}\big)=0\ .
\end{equation}
This completes the proof.
\end{proof}

The last part of this section is devoted to proving that functions $(L_n(u))_{0\leq n\leq N}$ are generically independent on $\mathcal{L}(N)$. Actually, it is sufficient to discuss the case $|\alpha|<<1$. For $\alpha$ small enough, we may consider the term $\alpha(K_u^{2n}(1)|1)$ as a perturbation, then we only need to study the independence of $F_n:=(K_u^{2n}(u)|u)$. Using the formula \eqref{formulaLXc2}, for any $0\leq n\leq N$,
$$F_n=J_{n+1}-\sum_{\substack{k+j=n\\ j\geq1,k\geq0}}F_kJ_j\ ,$$
with $J_n=(H_u^{2n}1|1)$. Assume there exists a sequence $c_n$ such that 
$$\sum_{n\geq0}c_nF_n= 0\ ,$$
we are to prove that $c_n\equiv0$.
Indeed, 
\begin{eqnarray*}
\sum_{n\geq0}c_nJ_{n+1}-\sum_{n\geq0}\sum_{\substack{k+j=n\\ j\geq1,k\geq0}}c_nF_kJ_j= \sum_n(c_n-\sum_{0\leq k\leq N-(n+1)}c_{n+k+1}F_k)J_{n+1}=0\ ,
\end{eqnarray*}
since all the $J_{n+1}$ are independent in the complement of a closed subset of measure $0$ of $\mathcal{L}(N)$ \cite{GGASENS}, then for every $n$,
$$c_n-\sum_{0\leq k\leq N-(n+1)}c_{n+k+1}F_k=0\ .$$
 Thus $c_N=c_{N-1}=\cdots=c_0=0$. 
 
 Finally, we now have $2N+1$ linearly independent and in involution conservation laws on a dense open subset of $2N+1$ dimensional complex manifold $\mathcal{L}(N)$, thus our system is completely integrable in the Liouville sense.
\section{Multiplicity and Blaschke product}
Recall the notation in section 2, there are two kinds of eigenvalues of $K_u$, some are the dominant eigenvalues of $K_u$, which are denoted as $\sigma_k\in\Sigma_K(u)$, while the others are the dominant eigenvalues of $H_u$ with multiplicities larger than $1$. Let us denote $u(t)$ as the solution of the $\alpha$--Szeg\H{o} equation with $\alpha\neq0$. Fortunately, we are able to show that for almost all $t\in\mathbb{R}$, the Hankel operator $H_{u(t)}$ has single dominant eigenvalues with multiplicities equal to $1$. In other words, for almost every time $t\in\R$, 
$$\mathrm{rk_d} K_{u(t)}=\mathrm{rk} K_{u(t)}=\mathrm{rk} K_{u_0}\ .$$
We call the phenomenon that $H_{u(t_0)}$ has some eigenvalue $\sigma$ with multiplicity $m\geq2$ as {\em crossing at $\sigma$ at $t_0$}. 

\subsection{The motion of singular values}
Let us first introduce the following Kato-type lemma.
\begin{lemma}[Kato]
Let $P(t)$ be a projector on a Hilbert space $\mathcal{H}$ which is smooth in $t\in I$, then there exists a smooth unitary operator $U(t)$, such that
\begin{align*}
P(t)=U(t)P(0)U^*(t)\ ,
\end{align*}
and
\begin{align}\label{unitaryc2}
\frac{d}{dt}U(t)=Q(t)U(t)\ ,\ U(0)=\mathrm{Id}\ ,
\end{align}
with $Q(t)=[P'(t), P(t)]$.
\end{lemma}
\begin{proof}
By simple calculus, we can prove $Q^*=-Q$. Since $P(t)$ is smooth in time, then by the Cauchy theorem for linear ordinary equations, $U(t)$ is well defined. The unitary property of $U(t)$ for every $t$ is a consequence of the anti self-adjointness of $Q$.
\begin{align*}
\frac{d}{dt}(U(t)^*U(t))=\frac{d}{dt}U^*U+U^*\frac{d}{dt}U=U^*Q^*U+U^*QU=0\ ,
\end{align*}
thus $U(t)^*U(t)=\mathrm{Id}$. On the other hand, 
\begin{align*}
\frac{d}{dt}(U(t)U(t)^*)=\frac{d}{dt}UU^*+U\frac{d}{dt}U^*=QUU^*-UU^*Q\ .
\end{align*}
It is obvious that $\mathrm{Id}$ is a solution to the linear equation $\frac{d}{dt}A=QA-AQ$ with $A(0)=\mathrm{Id}$, using the uniqueness of solutions, we have $U(t)U^*(t)=\mathrm{Id}$. We now prove that $U^*(t)P(t)U(t)$ does not depend on $t$.
 \begin{align*}
 \frac{d}{dt}(U^*(t)P(t)U(t))&= \frac{d}{dt}U^*(t)P(t)U(t)+U^*(t) \frac{d}{dt}P(t)U(t)+U^*(t)P(t) \frac{d}{dt}U(t)\\
 &=U^*Q^*PU+U^*P'U+U^*PQU\\
 &=U^*(P'+[P,Q])U\\
 &=U^*(P'-PP'-P'P)U=0
 \end{align*}
 where we have used $P^2=P$. This completes the proof.
\end{proof}

If $u_0\in H^s_+$ with $s > 1$, then the solution $u(t)$ of the $\alpha$--Szeg\H{o} equation \eqref{alszc2} is real analytic in $t$ valued in $H^s_+$. By the Lax pair for $K_u$, we know that the singular values of $K_u$ are fixed, with constant multiplicities. 
\begin{proposition}
Given any initial data $u_0\in H^s_+$ with $s>1$, let $u$ be the corresponding solution to the $\alpha$--Szeg\H{o} equation. Let $\sigma > 0$ be a singular eigenvalue of $K_u$ with multiplicity $m$, and write
\[\sigma_+>\sigma>\sigma_-\]
where $\sigma_+$, $\sigma_-$ are the closest singular values of $K_u$, possibly, $\sigma_+=+\infty$ or $\sigma_-=0$.
Then one of the following two possibilities occurs.
\begin{enumerate}
 \item $\sigma$ is a singular value of $H_{u(t)}$ with multiplicity $m + 1$ for every time $t$, and $u$ is a solution of the cubic Szeg\H{o} equation \eqref{szegoc2}.

  \item There exists a discrete subset $T_c$ of times outside of which the singular values of $H_{u(t)}$ in the interval $(\sigma_-, \sigma_+)$ are $\rho_1$, $\rho_2$ of multiplicity $1$, and $\sigma$ of multiplicity $m -1$ if $m \geq 2$, with
      \[\rho_1>\sigma>\rho_2\ ,\]
      and $\rho_1$, $\rho_2$ are analytic on every interval contained into the complement of $T_c$.
\end{enumerate}
\end{proposition}

\begin{proof}
Let us assume that $\sigma$ is a singular value of multiplicity $m+1$ of $H_{u(t_0)}$ for some time $t_0$. Then we may select $\delta > 0$ and $\epsilon > 0$ such that
\[\sigma_+>\sigma+\epsilon>\sigma>\sigma-\epsilon>\sigma_-\]
such that, for every $t \in [t_0-\delta,t_0+\delta]$, $\sigma^2-\epsilon$ and $\sigma^2+\epsilon$ are not eigenvalues of $H_{u(t)}^2$. Then we know that $H_{u(t)}^2$ has either $\sigma^2$ as an eigenvalue of multiplicity $m + 1$, or admits in $(\sigma^2-\epsilon,\sigma^2+\epsilon)$ two eigenvalues of multiplicity $1$, $\rho_1$, $\rho_2$ on both sides of $\sigma$. Set
\begin{equation}
\label{orthc2} P(t):=(2i\pi)^{-1}\int_{C(\sigma^2,\epsilon)}(z\mathrm{Id} -H_{u(t)}^2)^{-1}dz\ .
\end{equation}
We know that $P(t)$ is an orthogonal projector, depending analytically of $t \in (t_0 -\delta, t_0 + \delta)$, and that $P(t_0)$ is just the projector onto
\begin{equation*}
E(t_0) := \ker(H_{u(t_0)}^2-\sigma^2\mathrm{Id})\ .
\end{equation*}
Consider the selfadjoint operator
\begin{equation*}
A(t) := H_{u(t)}^2P(t)
\end{equation*}
acting on the $(m + 1)$-dimensional space $E(t) = \mathrm{Ran} P(t)$. Then its characteristic polynomial is
\begin{equation*}
\mathcal{P}(\lambda, t) = (\sigma^2-\lambda)^{m-1}(\lambda^2 + a(t)\lambda + b(t))\ ,
\end{equation*}
where $a$, $b$ are real analytic, real valued functions, such that
$$a^2 - 4b \ge 0\ .$$
Notice that the condition $a(t)^2 - 4b(t) = 0$ is precisely equivalent to the fact that $H_{u(t)}^2$ has $\sigma^2$ as an eigenvalue of multiplicity $m + 1$. Since
this function is analytic, it is either identically $0$, or different from $0$ for $0 < |t-t_0| < \delta $ and $\delta > 0$ small enough. Moreover, by the following perturbation analysis, the first condition only occurs if
$$(1|u(t)) = 0$$
for every $t \in (t_0 -\delta, t_0 +\delta)$. Since $(1|u)$ is a real analytic function of $t$, this would imply that it is identically $0$, whence $u$ is a solution of the cubic Szeg\H{o} equation. We now come back to the perturbation analysis, let $U(t)$ be a unitary operator given as in the Kato-type lemma above, denote
\begin{equation*}
B(t)=U^*(t)A(t)U(t)\ ,
\end{equation*}
then
\begin{equation*}
B(t_0)= \sigma^2\mathrm{Id} P(t_0)\ .
\end{equation*}
Let us calculate the derivative of $B$, we find
\begin{align*}
\frac{d}{dt}B(t)=\frac{d}{dt}\left(U^*(t)H_{u(t)}^2U(t)U^*(t)P(t)U(t)\right)=\frac{d}{dt}\left(U^*(t)H_{u(t)}^2U(t)P(t_0)\right)\ .
\end{align*}
Since $\frac{d}{dt}U(t)=Q(t)U(t)$ with $Q(t)=[P'(t),P(t)]$, then
\begin{equation*}
\frac{d}{dt}B(t)=U^*\Big(\frac{d}{dt}H_{u(t)}^2+[H_{u(t)}^2,Q(t)]\Big)UP(t_0)\ ,
\end{equation*}
using \eqref{laxhuc2},
\begin{align*}
\frac{d}{dt}H_{u(t)}^2=[B_u,H_u^2]-i\alpha(u|1)H_1H_u+i\alpha(1|u)H_uH_1\ .
\end{align*}
For any $h_1,h_2\in E(t_0)$, 
\begin{align*}
([B_u,H_u^2]h_1,h_2)+([H_u^2,Q]h_1,h_2)=0\ ,
\end{align*}
then
\begin{align*}
(\frac{d}{dt}B(t_0)h_1,h_2)=-i\alpha[(u(t_0)|1)(h_1|u(t_0))(1|h_2)-(1|u(t_0))(u(t_0)|h_2)(h_1|1)]\ .
\end{align*}
Denote by $v,w$ as the projections onto $E(t_0)$ of $1$ and $u$ respectively. If $(u(t_0)|1)\ne 0$, then the corresponding matrix under the base $(v,w)$ turns out to be
\begin{align*}
\begin{pmatrix}
-i\alpha(u|1)(v|w) & i\alpha(1|u)\|w\|^2\\
-i\alpha(u|1)\|v\|^2 & i\alpha(1|u)(w|v)
\end{pmatrix} 
\end{align*}
which has a negative determinant if $(u(t_0)|1)\neq0$. For the case $(u(t_0)|1)=0$ with $\frac{d^n}{dt^n}(u|1)(t_0)\neq0$ for some $n\in\mathbb{N}$, we only need to consider $\frac{d^{n+1}}{dt^{n+1}}(B(t))(t_0)$, 
\begin{align*}
\left(\frac{d^{n+1}}{dt^{n+1}}B(t_0)h_1,h_2\right)=-i\alpha\left[\Big(\frac{d^n}{dt^n}(u|1)\Big)(t_0)(h_1|u(t_0))(1|h_2)-\Big(\frac{d^n}{dt^n}(1|u)\Big)(t_0)(u(t_0)|h_2)(h_1|1)\right]\ ,
\end{align*}
with any $h_1,h_2\in E(t_0)$. It is similar as the case $n=0$. This completes the proof.
\end{proof}

Since $u(t)$ satisfying $(1|u(t))\equiv0$ would be a solution of the cubic Szeg\H{o} equation, which is well studied by G\'erard and Grellier \cite{GGASENS,GGINV,GGHW,GG2015AMS}. We assume $(1|u)$ is not identically zero in the rest of this article. From the discussion above, we have
\begin{corollary}
The dominant eigenvalues of $H_{u(t)}$ are of multiplicity $1$ for almost all $t\in\mathbb{R}$.
\end{corollary}

Recall the notation in section 2, by rewriting the conservation laws in Theorem~\ref{conservec2} as
\begin{equation}
L_n:=\Big(K_u^{2n}(u)\ |\ u\Big)-\alpha\Big(K_u^{2n}(1)\ |\ 1\Big)=\sum\limits_k \sigma_k^{2n} \left(\|u'_k\|^2-\alpha\|v'_k\|^2\right)\ ,
\end{equation}
we get the following conserved quantities
\begin{equation}
\ell_k:=\|u'_k\|^2-\alpha\|v'_k\|^2 \ .
\end{equation}

\begin{lemma}
Let $\alpha>0$. If there exists a crossing at $\sigma_k$ at time  $t=t_0$, then $\ell_k<0$.
\end{lemma}

\begin{proof}
Since there is a crossing at $\sigma_k$, then $\sigma_k\in\Sigma_H(u(t_0))$ with multiplicity $m\geq2$. Then
\begin{align*}
F_u(\sigma_k)=E_u(\sigma_k)\cap u^\perp=\left\{\frac{g}{D}H_u(u_k):\ g\in \mathbb{C}_{m-2}[z]\right\}\ .
\end{align*}
Hence, $u'_k=0$ while $v'_k\ne0$, since
\begin{eqnarray}
\Vert v'_k\Vert=\frac{(1,H_u(u_k))}{\Vert H_u(u_k)\Vert}=\frac{\Vert u_k\Vert}{\sigma_k}\neq0\ .
\end{eqnarray}
Thus $\ell_k=\|u'_k\|^2-\alpha\|v'_k\|^2 <0$ for $\alpha>0$.
\end{proof}

Here, we present an example to show the existence of crossing.
\begin{example}[Existence of crossing]
Let $u_0(z)=\frac{z-p}{1-pz}$ with $p\ne0$ and $|p|<1$, and $u$ be the corresponding solution to the equation
\begin{equation}
i\partial_t u=\Pi(|u|^2u)+(u|1)\ .
\end{equation}
It is obvious that $u_0\in \mathcal{L}(1)$ and $1\in\Sigma_H(u_0)$ with multiplicity 2, and 
$$L_1(u)=\Big(K_u^{2}(u)\ |\ u\Big)-\Big(K_u^{2}(1)\ |\ 1\Big)=-(1-|p|^2)<0\ .$$
Let us represent the Hamiltonian function $E=\frac14\Vert u\Vert_{L^4}^4+\frac12\vert (u|1)\vert^2$ under the coordinates
$$\rho_1,\rho_2,\sigma,\varphi_1,\varphi_2,\theta\ ,$$
\begin{align*}
E&=\frac14(\rho_1^4+\rho_2^4-\sigma^4)\\
&\qquad+\frac12\frac{\rho_1^2(\rho_1^2-\sigma^2)^2+\rho_2^2(\sigma^2-\rho_2^2)^2+2\rho_1\rho_2(\rho_1^2-\sigma^2)(\sigma^2-\rho_2^2)\cos(\varphi_1-\varphi_2)}{(\rho_1^2-\rho_2^2)^2}\\
&=\frac14+\frac12|p|^2\ .
\end{align*}
Notice that $\sigma=1$ and $\rho_1^2+\rho_2^2-\sigma^2=\Vert u\Vert_{L^2}^2=1$, then $\rho_1^2+\rho_2^2=2$. Set $I=\frac{\rho_1^2-\rho_2^2}{2}$, $\varphi=\varphi_1-\varphi_2$, then $\rho_1^2=1+I$ and $\rho_2^2=1-I$, thus we can rewrite $E$ as
\begin{align*}
E=\frac{1}{4}(1+2I^2)+\frac14(1+\sqrt{1-I^2}\cos(\varphi))\ .
\end{align*}
Thus
\begin{align*}
\frac{dI}{dt}&=-2\frac{\partial E}{\partial \varphi}=\frac12\sqrt{1-I^2}\sin(\varphi)\\
&=\pm\frac12\sqrt{-4I^2+(8|p|^2-5)I^2+4|p|^2(1-|p|^2)}\\
&=\pm\sqrt{(a-I^2)(b+I^2)}\ ,
\end{align*}
with $a,b$ satisfy
\begin{equation*}\left\{
\begin{split}
&a>0,b>0\ ,\\
&ab=|p|^2(1-|p|^2)\ ,\\
&a-b=2|p|^2-5/4\ .
\end{split}
\right.\end{equation*}

Recall the definition of Jacobi elliptic functions. The incomplete elliptic integral of the first kind $F$ is defined as
\[F(\varphi,k)\equiv\int_0^\varphi \frac{d\theta}{\sqrt{1-k^2\sin^2\theta}}\ ,\]
then the Jacobi elliptic function $sn$ and $cn$ are defined as follows,
\[sn(F(\varphi,k),k)=\sin\varphi\ ,\]
\[cn(F(\varphi,k),k)=\cos\varphi\ .\]
Then we may solve the above equation,
\begin{align*}
I(t)=\sqrt{a}\mathrm{cn}\left(\sqrt{a+b}\big(t-t_0\big)+F\Big(\frac\pi2,\sqrt{\frac{a}{a+b}}\Big),\sqrt{\frac{a}{a+b}}\right)\ .
\end{align*}
Therefore, there exists a discrete set of time $0\in T_c$, such that $I(t)=0$ for every $t\in T_c$. In other words, crossing happens at $t\in T_c$.
\end{example}
\subsection{Blaschke product}
We aim to show that the Blaschke products $\Psi(t)$ of $K_{u(t)}$ do not change their $\mathbb{S}^1$--orbits as times grows even before or after crossings.
\begin{proposition}\label{evoblaschc2}
For any open interval $\Omega$ contained into the complement of $T_c$, for any $\sigma_k\in\Sigma_K(u(t))$ with $t\in \Omega$,
\begin{equation}
\label{blaschkec2} K_{u(t)}u_k'(t)=\sigma_k\Psi_k(t)u'_k(t)\ .
\end{equation}
Then there exists a function $\psi_k(t): \Omega\to\mathbb{S}^1$, such that
\begin{equation}
\Psi_k (t)={\mathrm e}^{i\psi_k(t)}\Psi_k (0)\ ,\ t\in \Omega\ .
\end{equation}
\end{proposition}

\begin{proof}
Differentiating the above equation \eqref{blaschkec2} and using the Lax pair structure \eqref{laxc2}, one obtains
\begin{equation}
\label{diff1}[C_u,K_u](u'_k)+K_u\left(\frac{d u'_k}{dt}\right)=\sigma_k \dot \Psi_k u'_k+\sigma_k \Psi_k  \frac{d u'_k}{dt}\ .
\end{equation}
Recall $u'_k=P_k(u)$, where $P_k$ as \eqref{orthc2} by replacing $H_u$ with $K_u$, then
\begin{equation*}
\frac{d}{dt}P_k(t)=[C_u,P_k]\ .
\end{equation*}
Rewriting $\Pi(|u|^2u)=T_{|u|^2}(u)=(iC_u+\frac12K_u^2)u$, then the $\alpha$--Szeg\H{o} equation \eqref{alszc2} turns out to be
\[\frac{d u}{dt}=(C_u-\frac{i}{2}K_u^2)u-i\alpha(u|1)\ ,\]
then
\begin{align*}
\frac{d u'_k}{dt}&=(\frac{d}{dt}P_k)(u)+P_k(\frac{d u}{dt})\\
&=[C_u,P_k]u+P_k C_u u-\frac{i}{2} K_u^2P_k(u)-i\alpha(u|1)P_k(1)
\end{align*}
thus
\begin{equation}
\label{diff2}\frac{d u'_k}{dt}=-iT_{|u|^2} u'_k-i\alpha(u\ |\ 1)\frac{(1\ |\ u'_k)}{(u'_k\ |\ u'_k)}u'_k\ .
\end{equation}
Then \eqref{diff1} and \eqref{diff2} obtained above lead to
\begin{align*}
\left (\dot \Psi_k-i\left(\sigma_k^2+2\alpha \mathrm{Re}\big[\frac{(u\ |\ 1)(1\ |\ u'_k)}{(u'_k\ |\ u'_k)}\big]\right)\Psi_k \right )u'_k=-i[T_{|u|^2},\Psi_k](u'_k)\ .
\end{align*}

We claim that $$[T_{|u|^2},\Psi_k](u'_k)=0\ .$$
therefore
\begin{equation*}
\Psi_k (t)={\mathrm e}^{i(\sigma_k ^2t+\gamma_k (t))}\Psi_k (0)\ ,
\end{equation*}
where
\begin{align*}
\gamma_k (t)=2\alpha \int _0^t\frac{{\mathrm Re} [(u(t')\ |\ 1)(1\ |\ u'_k (t')]}{| u'_k (t')| ^2}\, dt'\ .
\end{align*}
It remains to prove the claim (one can also refer to \cite[Theorem 8]{GGHankel} for the proof). We first prove that, for any $\chi_p(z)=\frac{z-p}{1-pz}$ with $|p|<1$, 
$$[T_{|u|^2},\chi_p]f=0$$ 
for any $f\in F_u(\sigma_k)$ such that $\chi_pf\in F_u(\sigma_k)$. For any $L^2$ function $g$, 
$$[\Pi, \chi_p]g=(1-|p|^2)H_{1/(1-\overline{p}z)}(h)\ ,$$
where $\overline{(\mathrm{Id}-\Pi)g}=Sh$. Consequently, the range of $[\Pi,\chi_p]$ is one dimensional, directed by $\frac{1}{1-\overline{p}z}$. In particular, $[T_{|u|^2},\chi_p]f$ is proportional to $\frac{1}{1-\overline{p}z}$.
Since
\begin{align*}
([T_{|u|^2},\chi_p]f|1)&=(T_{|u|^2}\chi_pf-\chi_pT_{|u|^2}f|1)\\
&=(\chi_pf|H_u^2(1))-(\chi_p|1)(H_u^2f|1)\\
&=(H_u^2(\chi_pf)|1)-(\chi_p|1)(H_u^2f|1)\\
&=(\chi_pf-(\chi_p|1)f|u)(u|1)\ ,
\end{align*}
We used \eqref{KH} to gain the last equality. Since $\chi_pf-(\chi_p|1)f\in F_u(\sigma_k)$ is orthogonal to $1$, by Proposition~\ref{actionc2}, $\chi_pf-(\chi_p|1)f\in E_u(\sigma_k)$, hence $\chi_pf-(\chi_p|1)f\in F_u(\sigma_k)$ is orthogonal to $u$. This proves that $[T_{|u|^2},\chi_p]f=0$.  
\end{proof}

Therefore, we have
\begin{corollary}
$$\mathrm{rk} K_{u(t)}=\mathrm{rk_d} K_{u(t)}=\mathrm{rk} K_{u_0},\  a.e.\ t<\infty.$$
\end{corollary}

 We know that $\Psi _k(t)$ is defined for every $t$ in an open subset $\Omega $ of $\R $ consisting of the complement of  a discrete closed subset, corresponding to crossings at $\sigma _k^2$. Furthermore, by Proposition~\ref{evoblaschc2}, on each connected component of $\Omega $, the zeroes of $\Psi _k(t)$ are constant. Together with the following property, $\Psi_k(t)$ never changes it orbit even after the crossings.
\begin{proposition}\label{evoblaschke}
 For every time $t$ such that $\Psi _k(t)$ is defined, the zeroes of $\Psi _k(t)$ are the same.
 \end{proposition}
 
 \begin{proof}
The proposition is a consequence of the following lemma.
 \end{proof}
\begin{lemma}
There exists an analytic function $\Psi _k^\sharp $ defined in a neighborhood $\Omega '$ of  $\Omega ^c $ and valued into rational functions, 
and, for every $t\in \Omega \cap \Omega '$, there exists $\beta (t)\in \T $ such that
$$\Psi _k(t, z)={\rm e}^{i\beta _k(t)}\Psi _k^\sharp (t,z)\ .$$
\end{lemma}
\begin{proof}
 Since $\sigma _k^2$ is an eigenvalue of constant multiplicity $m$ of $K_{u(t)}^2$, the orthogonal projector $P_k(t)$ onto $F_{u(t)}(\sigma _k)$ is an analytic function of $t\in \R $. Consequently, the vector
$$v'_k(t):=P_k(t)(1)$$
depends analytically on $t$. Furthermore, $v'_k(t)$ is not $0$ if $t\not \in \Omega $. Indeed, from the description of $F_u(\tau)$ provided by Proposition~\ref{actionc2} when $\tau $ is a singular value associated to the pair $(H_u,K_u)$, we observe that, if $\tau $ is $H$ dominant, the space $F_u(\tau )$ is not orthogonal to $1$. Consequently, we can define, for $t$ in a neighborhood  $\Omega '$ of $\Omega ^c$,
$$\Psi _k^\sharp (t,z):=\frac{K_{u(t)}(v'_k(t))(z)}{\sigma _kv'_k(t,z)}$$
as an analytic function of $t$ valued into rational functions of $z$. On the other hand, if $t\in \Omega $, Proposition~\ref{actionc2} shows that 
$$F_{u(t)}(\sigma _k)\cap u(t)^\perp =E_{u(t)}(\sigma _k)=F_{u(t)}(\sigma _k)\cap 1^\perp \ ,$$
therefore $v'_k(t)$ is collinear to $u'_k(t)$,
$$v'_k(t)=(1\vert u'_k(t))\frac{u'_k(t)}{\Vert u'_k(t)\Vert ^2}\ .$$
Since, from the definition of $\Psi _k(t)$,
$$K_{u(t)}(u'_k(t))=\sigma _k\Psi _k(t)u'_k(t)\ ,$$
we infer that there exists an analytic $\beta _k$ on $\Omega \cap \Omega '$ valued into $\T $ such that
$$K_{u(t)}(v'_k(t))=\sigma _k{\rm e}^{-i\beta _k(t)}\Psi _k(t)v'_k(t)\ .$$
This completes the proof.
 \end{proof}

\section{Necessary condition of norm explosion} In this section, let $u(t)$ be the solution of $\alpha$--Szeg\H{o} equation \eqref{alszc2} with initial data $u_0\in\mathcal{L}(N)$, $N\in\mathbb{N}^+$, $u^\infty =\lim u(t_n)$ for the weak * topology of $H^{1/2}$, for some sequence $t_n$ going to infinity. To study the large time behavior of solutions, it is equivalent to study the rank of the shifted Hankel operator $K_u$. 
\begin{lemma}
The solution $u(t)$ to the $\alpha$--Szeg\H{o} equation will stay in a compact subset of $\mathcal{L}(N)$ if and only if for all the adherent values $u^\infty$ of $u(t)$ at infinity,
\begin{equation}
\mathrm{rk}K_{u^\infty}=\mathrm{rk}K_{u_0}\ .
\end{equation}
\end{lemma}
\begin{proof}
By the explicit formula of functions in $\mathcal{L}(N)\subset H^s$ for every $s$ in Theorem~\ref{kroneckerc2}, $\mathrm{rk}u(t)=N$ if and only if 
$$u(z)=\frac{A(z)}{B(z)}$$
with $A,B\in\mathbb{C}_N[z],A\wedge B=1, \deg( A)= N \text{ or }
\deg( B)= N, B^{-1}(\{0\})\cap \overline{\mathbb D}=\emptyset$.

Then a sequence of $(u_n)_n$ is in a relatively compact subset of $\mathcal{L}(N)$ unless one of the poles of $u_n$ approaches the unit disk $\mathbb{D}$, then the corresponding limit $u(z)$ will be in some $\mathcal{L}(N')$ with $N'<N$.
\end{proof}

We first present a necessary condition of the norm explosion for any $\alpha\in\mathbb{R}\setminus\{0\}$.

\begin{theorem}\label{necc2}
If $\mathrm{ rk} K_{u^\infty}<\mathrm{rk} K_{u_0}$, then there exists some $k$ such that $\ell_k(u_0)=0$.
\end{theorem}

\begin{corollary}\label{cornecc2}
If $\alpha<0$, for any $N\in\mathbb{N}^+$, given initial data $u_0\in\mathcal{L}(N)$, then the solution to the $\alpha$--Szeg\H{o} equation stays in a compact subset of $\mathcal{L}(N)$.
\end{corollary}

\begin{proof}[Proof of Corollary~\ref{cornecc2}]
Since $\alpha<0$, then $\ell_k:=\|u'_k\|^2-\alpha\|v'_k\|^2>0$, due to Theorem~\ref{necc2}, $\mathrm{rk} K_{u^\infty}\equiv \mathrm{rk} K_{u_0}$.
\end{proof}

\begin{proof}[Proof of Theorem~\ref{necc2}]
Assume $\mathrm{rk} K_{u^\infty}< \mathrm{rk} K_{u_0}$, then there exists some $k$ such that $\dim F_{u^\infty}(\sigma_k)<\dim F_{u_0}(\sigma_k)=m$. We are to prove $\|u'^\infty_k\|^2=0$ and $\|v'^\infty_k\|^2=0$.
\begin{itemize}
  \item $\|u'^\infty_k\|^2=0$.
  
There exists a time dependent Blaschke product $\Psi_k$ of degree $m-1$ such that
\begin{equation}
\label{psic2}K_{u(t_n)}^2(u'_k (t_n))=\sigma_k ^2u'_k (t_n)\ ,\ K_{u(t_n)}(u'_k (t_n))=\sigma_k \Psi_k (t_n)u'_k (t_n)\ ,
\end{equation}
By Proposition~\ref{evoblaschke}, any limit point of $\Psi_k (t)$ as $t$ goes to $\infty $ is of degree $m-1$ as well. Since $u'_k (t_n)$ is bounded in $L^2_+$, up to a subsequence it converges weakly to some $u'^\infty_k\in L^2_+$.
Passing to the limit in the identities \eqref{psic2}, we get
\begin{equation}
\label{psiinfc2}K_{u^\infty}^2(u'^\infty_k)=\sigma_k ^2u'^\infty_k\ ,\ K_{u^\infty }(u'^\infty_k)=\sigma_k \Psi_k^\infty u'^\infty_k\ ,
\end{equation}
where $\Psi_k ^\infty $ is a Blaschke product of degree $m-1$. The latter identities \eqref{psiinfc2} show that $u'^\infty_k$ and $\Psi_k^\infty u'^\infty_k$ belong to $F_{u^\infty}(\sigma_k)$, hence, if $u'^\infty_k$ is not zero, the dimension of $F_{u^\infty} (\sigma_k )$ is at least $m$. Indeed, if we write $\Psi_k ^\infty={\mathrm e}^{-i\psi}\frac{P(z)}{D(z)} $, then
\begin{equation}
F_{u^\infty}(\sigma_k)=\left\{\frac{f}{D(z)}u'^\infty_k,\ f\in\mathbb{C}_{m-1}[z]\ \right\}\ .
\end{equation}

  \item $\|v'^\infty_k\|^2=0$.

Recall the structure of $F_u(\sigma_k)$ with $\sigma_k\in \Sigma_K(u)$ in Proposition~\ref{actionc2}, the orthogonal projection of $1$ onto the space $F_u(\sigma_k)$, $v'_k$ can be represented as
\begin{equation*}
v'_k=\Big(1\ |\ \frac{u'_k}{\|u'_k\|}\Big)\frac{u'_k}{\|u'_k\|}\ .
\end{equation*}
If $v'^\infty_k\neq0$, since $\|v'_k\|=\big|(1\ |\ \frac{u'_k}{\|u'_k\|})\big|$ ,
thus $\frac{u'_k}{\|u'_k\|}\rightharpoonup v$ in $L^2$ with $v\neq0$. Using the strategy in the first step above by replacing $u'_k$ by $\frac{u'_k}{\|u'_k\|}$, we have $\dim F_{u^\infty}(\sigma_k)=m$.
\end{itemize}
\end{proof}

\section{Large time behavior of the solution for the case $\alpha>0$}
In this section, we prove  for any $N$, there exist solutions in $\mathcal{L}(N)$ which admit an exponential on time norm explosion.
\begin{theorem}\label{suffnecL1}
For $\alpha>0$, $u_0\in H^s_+$ such that $\Sigma_K(u_0)=\{\sigma\}$ with multiplicity $k=\mathrm{rk} K_{u_0}$. Then $\Vert u(t)\Vert_{H^s}$ grows exponentially on time,
$$\|u(t)\|_{H^s}\simeq \mathrm{e}^{C_{\alpha}(2s-1)|t|}\ ,$$
 if and only if
\begin{equation}
L_1(u):=(K_u^2(u)|u)-\alpha(K_u^2(1)|1)=0\ .
\end{equation}
\end{theorem}

Let $u_0$ as in the theorem above. If $u_0$ is not a Blaschke product, we have
\begin{equation*}
\Sigma _H(u_0)=\{ \rho _1,\rho _2\} \ ,\ \rho _1>\sigma >\rho _2\ .
\end{equation*}
Using the results by G\'erard and Grellier \cite{GGHankel}, we have the explicit formula for the solution $u$ as
\begin{equation}
u(t,z)=\frac{\triangle_{11}-\triangle_{21}}{\det(\mathcal{C}(z))}{\mathrm e}^{-i\varphi_1}+\frac{\triangle_{22}-\triangle_{12}}{\det(\mathcal{C}(z))}{\mathrm e}^{-i\varphi_2}\ ,
\end{equation}
with $\triangle_{jk}$ as the minor determinant of $\mathcal{C}(z)$ corresponding to line $k$ and column $j$, and
\begin{equation*}
\mathcal{C}(z)=\begin{pmatrix}
\frac{\rho_1-\sigma z \Psi {\mathrm e}^{-i\varphi_1}}{\rho_1^2-\sigma^2} & \frac{\rho_2-\sigma z \Psi {\mathrm e}^{-i\varphi_2}}{\rho_2^2-\sigma^2}  \\
\frac{1}{\rho_1} & \frac{1}{\rho_2}
\end{pmatrix} 
\end{equation*}
Then
\begin{equation*}
u(t,z)=\frac{(\frac{1}{\rho_2}-\frac{\rho_2-\sigma z\Psi {\mathrm e}^{-i\varphi_2}}{\rho_2^2-\sigma^2}){\mathrm e}^{-i\varphi_1}+(\frac{\rho_1-\sigma z\Psi {\mathrm e}^{-i\varphi_1}}{\rho_1^2-\sigma^2}-\frac{1}{\rho_1}){\mathrm e}^{-i\varphi_2}}{\frac{1}{\rho_2}(\frac{\rho_1-\sigma z \Psi {\mathrm e}^{-i\varphi_1}}{\rho_1^2-\sigma ^2})-\frac{1}{\rho_1}(\frac{\rho_2-\sigma z \Psi {\mathrm e}^{-i\varphi_2}}{\rho_2^2-\sigma^2})}\ .
\end{equation*}

An interesting fact is that $u$ is under the form
\begin{equation*}
u(t,z)=b(t)+\frac{c'(t) z\Psi(t,z)}{1-p'(t)z\Psi(t,z)}\ ,
\end{equation*}
where $b\ ,p',\ c'\in\mathbb{C}$.
Since $\Psi(t,z)={\mathrm e}^{i\psi(t)}\chi(z)$ with $\chi$ as a time independent Blaschke product, we then rewrite
\begin{equation}\label{starstar}
u(t,z)=b(t)+\frac{c(t) z\chi(z)}{1-p(t)z\chi(z)}\ .
\end{equation}

\begin{lemma}
Let $\chi$ be a time-independent Blaschke product. A function $u\in C^\infty(\R,H^s_+)$ with $s>\frac12$ is a solution of the $\alpha$--Szeg\H{o} equation,
\begin{align*}
i\partial_t u=\Pi(|u|^2u)+\alpha(u|1)\ ,
\end{align*}
if and only if 
\begin{equation*}
\widetilde{u}(t,z):=u(t,z\chi(z))
\end{equation*}
satisfies the $\alpha$--Szeg\H{o} equation.
\end{lemma}

\begin{proof}
 First of all, $z\chi(z)\in C^\infty_+(\mathbb{S}^1)$, then $(z\chi(z))^n\in C^\infty_+(\mathbb{S}^1)$ for any $n$, so that $u\in H^s_+$ implies $\widetilde{u}\in H^s_+$.
 Assume $u$ is a solution of the $\alpha$--Szeg\H{o} equation, it is equivalent to 
 \begin{equation}\label{Fourierszego}
 i\partial_t \hat{u}(t,n)=\sum\limits_{p-q+r=n}\hat{u}(t,p)\overline{\hat{u}}(t,q)\hat{u}(t,r)+\alpha\hat{u}(t,0)\delta_{n0}\ ,\ \forall n\geq0\ .
 \end{equation}
 Since
\begin{equation*}
\Pi(|u(z\chi(z))|^2u(z\chi(z)))=\sum\limits_{p-q+r\ge0}\hat{u}(p)\overline{\hat{u}}(q)\hat{u}(r)(z\chi(z))^{p-q+r}\ ,
\end{equation*}
we obtain that $\widetilde{u}$ satisfies the $\alpha$--Szeg\H{o} equation.

Conversely, assume $\widetilde{u}$ satisfies the $\alpha$--Szeg\H{o} equation, then we have
 \begin{equation}\label{star}
 i\partial_t \hat{u}(n)(z\chi(z))^n=\sum\limits_{p-q+r\geq0}\hat{u}(p)\overline{\hat{u}}(q)\hat{u}(r)(z\chi(z))^{p-q+r}+\hat{u}(0)\ .
 \end{equation}
Identifying the Fourier coefficients of $0$ mode of both sides, we get equation \eqref{Fourierszego} with $n=0$. Then withdraw this quantity from both sides of \eqref{star} and simplify by $z\chi(z)$. Continuing this process, we get all the equations \eqref{Fourierszego} for every $n$.
\end{proof}

\begin{lemma}\label{estBlaschke}
Let $\Psi $ be a Blaschke product of finite degree $d$ and $s\in [0,1)$.  There exists $C_{\Psi ,s}>0$ such that,  for every $p\in \mathbb{D} $,
$$\left \Vert \frac{1}{1-p\Psi }\right \Vert _{H^s(\mathbb{S}^1)}\ge \frac{C_{\Psi ,s}}{(1-\vert p\vert )^{s+\frac 12}}\ .$$
\end{lemma}
\begin{proof}
It is a classical fact that, for every $u\in H^s_+(\mathbb{S} ^1)$, for every $s\in [0,1)$,
$$\Vert u\Vert _{H^s(\mathbb{S} ^1)}^2 \simeq \int _{\mathbb{D} }\vert u'(z)\vert ^2(1-\vert z\vert ^2)^{1-2s}\, dL(z)\ ,$$
where $L$ denotes the bi-dimensional Lebesgue measure.

Let $p\in \mathbb{D} $ close to the unit circle and 
$$\omega :=\frac{p}{\vert p\vert }\ .$$
Since $\Psi $ is a Blaschke product of finite degree $d$, the equation 
$$\omega \Psi (z)=1$$
admits $d$ solutions on the circle. Moreover, these solutions are simple. Indeed, writing
$$\Psi (z)={\rm e}^{-i\psi} \prod _{j=1}^d \frac{z-p_j}{1-\overline p_jz}\ ,\ \vert p_j\vert <1\ ,$$
we have, for every $z\in \mathbb{S} ^1$,
$$\frac{\Psi '(z)}{\Psi (z)}=\frac 1z \sum _{j=1}\frac{1-\vert p_j\vert ^2}{\vert z-p_j\vert ^2}\ne 0\ .$$
Let $\alpha $ be such a solution. 
 For every $z$ such that $$\vert z-\alpha \vert \le  (1-\vert p\vert ),$$ we have, if $1-\vert p\vert $ is small enough,
$$\vert 1-p\Psi (z)\vert =\vert 1-p\Psi (\alpha )- p\Psi '(\alpha )(z-\alpha )+O(\vert z-\alpha \vert ^2)\vert \le C( 1-\vert p\vert ).$$
Therefore
\begin{align*}
\left \Vert \frac{1}{1-p\Psi }\right \Vert _{H^s(\mathbb{S}^1)}^2&\ge A_s\int _{\mathbb{D} \cap \{ \vert z-\alpha \vert \le (1-\vert p\vert )\} }\left \vert \frac {\Psi '(z)}{(1-p\Psi (z))^2}\right \vert ^2(1-\vert z\vert ^2)^{1-2s}\, dL(z)\\
 &\ge B_{\Psi ,s}(1-\vert p\vert )^{-4}\int _{\mathbb{D} \cap \{ \vert z-\alpha \vert \le (1-\vert p\vert )\} }(1-\vert z\vert ^2)^{1-2s}\, dL(z)\\
 &\ge \frac{C_{\Psi ,s}^2 }{(1-\vert p\vert )^{2s+1}}\ .
\end{align*}
\end{proof}

Let us turn back to prove the theorem.\\
\begin{proof}[Proof of Theorem~\ref{suffnecL1}]
Recall that
\begin{align*}
L_1(u)&=\big(K_u^2(u)\ |\ u\big)-\alpha\big(K_u^2(1)\ |\ 1\big)\\
&=\frac12\big(\|u\|_{L^4}^4-\|u\|_{L^2}^4\big)-\alpha\big(\|u\|_{L^2}^2-|(u\ |\ 1)|^2\big)\ .
\end{align*}
Since $\chi(z)$ is an inner function, we have
$$(\widetilde{u}\ |\ \widetilde{v})=(u|v)\ ,\forall\ u,v\ ,$$
thus
\begin{align*}
(\widetilde{u}|1)=(u|1)\ , \Vert \widetilde{u}\Vert_{L^2}=\Vert u\Vert_{L^2}\ ,
\end{align*}
and since $$\widetilde{u^2}=(\widetilde{u})^2\ ,$$
then
\begin{align*}
\Vert \widetilde{u}\Vert_{L^4}=\Vert u\Vert_{L^4}\ .
\end{align*}
 As a consequence, $L_1(u)=L_1(\widetilde{u})=0$. 

 The solution $\widetilde{u}$ is under the form \eqref{starstar}, 
\begin{align*}
u(t,z)=b(t)+\frac{c(t) z\chi(z)}{1-p(t)z\chi(z)}=b-\frac{c}{p}+\frac{c}{p}\frac{1}{1-pz\chi(z)}\ ,
\end{align*} 
 thus
\begin{align*}
\Vert u\Vert_{H^s}&\simeq\vert c\vert\Vert\frac{1}{1-pz\chi(z)}\Vert_{H^s}\\
&\geq C_{\chi,s}\frac{|c|}{(1-|p|)^{s+1/2}}\ ,
\end{align*}
 where we used Lemma~\ref{estBlaschke}.
 Using the result in \cite[Theorem 3.1]{XU2014APDE} and its proof, we have 
 $$\frac{|c|}{(1-|p|)^{s+1/2}}\simeq (1-|p|)^{-s+1/2}\simeq\mathrm{e}^{C_{\alpha}(2s-1)|t|}\ .$$
Therefore, $\widetilde{u}$ admit an exponential on time growth of the Sobolev norm $H^s$ with $s>\frac12$. The proof is complete.
\end{proof}

\section{Perspectives}
The main purpose of this work is to study the dynamics of the general solutions of the $\alpha $--Szeg\H{o} equation \eqref{alszc2}. We have already observed the weak turbulence by considering some special rational data. We proved the existence of data with exponential in time growth, a natural question is about the genericity of data with such a high growth. Besides, an important open problem is to gain new informations on the solutions with infinite rank. 

Another interesting question is about the cubic Szeg\H{o} equation with other perturbations, for example, consider a Hamiltonian function 
$$E(u)=\frac14\Vert u\Vert_{L^4}^4+\frac 12 F(\vert (u\vert 1)\vert ^2)\ ,$$
with a non linear function $F$. In this case, we still have one Lax pair $(K_u, C_u)$ while the conservation laws we found no longer exist. The question is to study the integrability and also the existence of turbulent solutions of this new Hamiltonian system.


\end{document}